\documentclass[10pt]{article}

\usepackage[usenames,dvipsnames]{color}

\usepackage{amsmath,amssymb,amsthm}
\usepackage{graphicx}
\usepackage{tikz}
\usepackage{enumerate}
\usepackage[T1]{fontenc}
\usepackage[utf8]{inputenc}
\usepackage{authblk}
\usepackage{color}

\newcommand{\per}{{\rm{per}}}
\newcommand{\perrank}{{\rm{perrank}}}
\newcommand{\ppr}{{\rm{ppr}}}
\newcommand{\hs}{\hspace{1mm}}
\newcommand{\cyc}{{\rm{cyc}}}

\newtheorem{prop}{Proposition}
\newtheorem{defi}{Definition}
\newtheorem{thm}{Theorem}
\newtheorem{lem}{Lemma}

\newtheorem{fact}{Fact}

\title{On the Principal Permanent Rank Characteristic Sequences of Graphs and Digraphs\thanks{Research supported in part by NSF grant DMS-1427526, ``The Rocky Mountain - Great Plains Graduate Research Workshop in Combinatorics''.}
}

\author{Keivan Hassani Monfared\thanks{Department of Mathematics and Statistics, University of Calgary, k1monfared@gmail.com} 
\and Paul Horn\thanks{Department of Mathematics, University of Denver, paul.horn@du.edu}
\and Franklin H. J. Kenter\thanks{Department of Computational and Applied Mathematics, Rice University, franklin.h.kenter@rice.edu}
\and Kathleen Nowak\thanks{Department of Mathematics, Iowa State University, knowak@iastate.edu}
\and John Sinkovic\thanks{Department of Combinatorics and Optimization, University of Waterloo, johnsinkovic@gmail.com}
\and Josh Tobin\thanks{Department of Mathematics, University of California - San Diego, rjtobin@ucsd.edu}}

\begin{document}
\maketitle 

\begin{abstract}
%The permanent rank of a matrix is the size of the largest square submatrix with nonzero permanent. 
%Similarly, 
The principal permanent rank characteristic sequence
%is the largest principal square submatrix with nonzero permanent. 
is a binary  sequence $r_0 r_1 \ldots r_n$ where $r_k = 1$ if there exists a principal square submatrix of size $k$ with nonzero permanent and $r_k = 0$ otherwise, and $r_0 = 1$ if there is a zero diagonal entry. 

A characterization is provided for all principal permanent rank sequences obtainable by the family of nonnegative matrices as well as the family of nonnegative symmetric matrices. Constructions for all realizable sequences are provided. 
%The main approach is to relate the principal permanent rank of matrices to properties of graphs. In particular, we treat each matrix as a weighted graph and investigate  the relation of generalized cycles as subgraphs to the principal permanent rank sequence.
Results 
%for tensor products and 
for skew-symmetric matrices are also included.

\end{abstract}


\section{Introduction}

The \emph{principal rank characteristic sequence problem} introduced by Brualdi, Deaett, Olesky and van den Driessche asks \cite{brualdi12}: 
\begin{quotation}
\noindent Given a binary sequence $r_0 r_1 \ldots r_n$ of length $n+1$, is there an $n \times n$ matrix $A$ such that  $r_k = 1$ if and only if there is a principal submatrix of rank $k$?
\end{quotation}
This problem is a simplified form of the more general \emph{principal assignment problem} (see for example \cite{holtz02}).

Recently, several groups have studied the principal rank characteristic sequence problem with different variations. For real matrices, Brualdi, et. al, characterize all realizable sequences with $n \le 6$ and all realizable sequences beginning $010\ldots$ for $7 \le n \le 10$. They also provide several forbidden subsequences \cite{brualdi12}. Barrett, et. al. characterize all allowable sequences over fields with characteristic 2 and also provide additional results for other fields \cite{barrett14}. Additionally, in \cite{butler}, the authors study a variation, the \emph{enhanced principal rank sequence}, which differentiates whether ``some'' or ``all'' of the principal minors of order $k$ have rank $k$ where they characterize all such realizable sequences for real matrices of order $n\le 5$.

Our focus will be the permanent, $\per(A)$, instead of the rank or determinant. Recall the definition of the permanent:
\begin{defi}[From \cite{minc78}]
	The \emph{\textbf{permanent}} of an $n\times n$ matrix $A$ is defined to be the sum of all diagonal products of $A$. That is,
	\[\displaystyle \per(A)=\sum_{\sigma\in S_n} \left( \prod_{i=1}^{n} a_{i\,\sigma(i)}\right)\]
\end{defi}
while
\[\displaystyle \det(A)=\sum_{\sigma\in S_n} \left({\rm{sgn}}(\sigma) \prod_{i=1}^{n} a_{i\,\sigma(i)}\right)\]
where $\rm{sgn}(\cdot)$ is the sign of the permutation.
Hence, in some sense, the permanent can be viewed as a variation of the determinant.  
%where $S_n$ is the set al all permutations on $1, \ldots  n$. Hence, the permanent can be viewed as variation of the determinant. 

Note that determining whether there is a principal submatrix of rank $k$ is equivalent to seeing if there is a principal submatrix of size $k$ with nonzero determinant (see \cite{brualdi12}).
Therefore, in a similar fashion, one can define the \emph{permanent rank}:
\begin{defi}[From \cite{yu99}]
	The \emph{\textbf{principal permanent rank}}  of a matrix $A$, denoted $\perrank (A)$ is defined to be the size of the largest square submatrix with nonzero permanent. 
\end{defi}
%In fact, Yu first defined and studied the permanent rank in relation to the Alon-Jaeger-Tarsi conjecture \cite{yu99}.

We study the principal permanent rank characteristic sequence defined as follows.

\begin{defi}
Given an $n\times n$  matrix $A$, the \emph{\textbf{principal characteristic permanent rank sequence}} of $A$ (abbreviated ppr-sequence of $A$ or $\ppr(A)$) is defined as $r_0 r_1 r_2 \ldots r_n$ where for $1 \leq k \leq n$
\[ r_k= \begin{cases}
1 & \text{if $A$ has a principal submatrix of size $k$ with nonzero permanent, and }\\
0 & \text{otherwise,}
\end{cases}\]
while $r_0 = 1$ if and only if $A$ has a zero on its main diagonal.
\end{defi}

%{\color{blue} Maybe boldface the terms being introduced in the definitions? Nonitalic in italic looks weird to me but it's not a sticking point. }

%studied the properties of the \emph{permanent rank} of a matrix which is the largest $k$ for which there is  a square $k \times k$ submatrix with nonzero permanent \cite{yu99}.

Naturally, in this paper, we introduce the \emph{principal permanent rank sequence problem}:

\begin{quotation}
	\noindent Given a binary sequence $r_0 r_1 \ldots r_n$, when  is there an $n \times n$ matrix $A$ such that  $ppr(A) = r_0 r_1 \ldots r_n$? 
\end{quotation}

Our contribution is to answer this question and to fully characterize which sequences can be realized for various families of real matrices including
\begin{itemize}
\item nonnegative matrices (Section \ref{sectiondirected}, Theorem \ref{onenonsymmetric}),
\item symmetric nonnegative matrices (Section \ref{sectionundericted}, Theorem \ref{fuzzyegg}), and
\item skew-symmetric matrices whose underlying graph is a tree (Section \ref{skew}, Theorem \ref{thm.skewtree}).
\end{itemize}
Additionally, for each characterization, we provide a construction that produces a realization for any realizable sequence.

\section{Preliminaries}

Our main approach is to exploit the duality between matrices and graphs. Throughout, we will consider graphs, both directed and undirected and with or without loops. However, we will not consider graphs with multiple edges (see Proposition \ref{zerononzero}).

Let $[n] = \{1, \ldots, n\}$. For a (directed) graph $G$ on $n$ vertices, $V(G) = [n]$, and $\alpha \subseteq [n]$, the graph $G[\alpha]$ is the induced subgraph of $G$ on vertices in $\alpha$. For an $n\times n$ matrix $A$ and $\alpha \subseteq [n]$, $A[\alpha]$ denotes the principal submatrix of $A$ from rows and column indexed by $\alpha$. The zero--nonzero pattern of $A$ is a $(0,1)$-matrix $B$ of the same order where $B_{ij}=1$ if and only if $A_{ij}\neq 0$. Also,  the \emph{underlying graph} of a matrix $A$ is the graph $G$ whose adjacency matrix is the zero--nonzero pattern of $A$. Note that $G$ is undirected if and only if the zero--nonzero pattern of $A$ is symmetric.

%If $A$ is symmetric, we take $G$ to be undirected otherwise $G$ is directed.
%The following proposition shows that the ppr-sequence of a nonnegative matrix and its zero--nonzero pattern are equal. 
The following proposition shows that the ppr-sequence of a nonnegative matrix and its zero--nonzero pattern are one and the same. Thus, for a nonnegative matrix, we will focus our attention on its underlying graph. 
%Therefore, for our study, it is not necessary to consider multiple edges.

%{\color{blue} My reasoning for the change above: Nonnegative matrices are not necessarily integral. The matrix more naturally corresponds to a weighted graph than one with multiple edges in my opinion. }

\begin{prop} \label{zerononzero}
	Let $B$ be the zero--nonzero pattern of an $n\times n$ nonnegative matrix $A$. Then $\ppr(A) = \ppr(B)$.
\end{prop}
\begin{proof}
	Let $\ppr(B) = q_0 q_1 \ldots q_{n}$ and $\ppr(A) = r_0 r_1 \dots r_{n}$. First, by definition, $a_{ii} = 0$ if and only if $b_{ii} = 0$ for $i \in [n]$. Thus, $r_0 = q_0$. 

	Now fix $k \in [n]$ and let $\alpha = \{i_1,i_2, ..., i_k\} \subseteq [n]$. Note that every term in both $\per(A[\alpha])$ and $\per(B[\alpha])$ is nonnegative. Next let $S_\alpha$ be the set of all permutations on $\alpha$. Then for $\sigma \in S_\alpha$, 
	\[
	\prod_{j = 1}^k a_{i_j, \sigma(i_j)} > 0 \hs \text{ if and only if } \hs \prod_{j = 1}^k b_{i_j, \sigma(i_j)} > 0
	\]
	and vice versa. Thus,
	\[
	\per(A[\alpha]) = \sum_{\pi \in S_\alpha}\prod_{j = 1}^k a_{i_j,\pi(i_j)} > 0\]
	if and only if \[ \per(B[\alpha]) = \sum_{\pi \in S_\alpha}\prod_{j = 1}^k b_{i_j,\pi(i_j)} > 0.
	\]
	Hence, $q_k = r_k$, for $k=1,2,\ldots, n$.
\end{proof}

 It is well known that various graph properties are captured by the permanent rank of matrices describing the graph. Such properties include the size of a largest \textit{generalized cycle} and the size of the largest perfect matching in the graph (see \cite{monfared12} and the references therein). Let us formally define a generalized cycle.

\begin{defi}
 A \emph{\textbf{generalized cycle}} of size $k$ is a permutation, $\pi_C$, on a subset of $k$ vertices, $C$, such that $i \pi_C(i)$ is a directed edge (or a loop if $i = \pi_C(i)$) for all $i \in C$.
\end{defi}

%{\color{blue} We need a '.' in the definition above. Also are these edges directed or not? You can define them as directed but then you need to be more careful in your permanent calculation I think.}

Observe that for a (directed) graph $G$, $C \subset V(G)$ supports a generalized cycle if there is a collection of edges within $G[C]$, such that every component of the subgraph induced on those edges has a Hamiltonian cycle. A generalized cycle can be viewed as both a permutation or a subset of edges. Here, a bi-directed edge (or undirected edge) can be seen as a 2-cycle.  
%As an abuse of notation, 
With this clear bijection, we will refer to such a collection of cycles also as a ``generalized cycle.''

%\begin{defi}
%Given an $n\times n$ real matrix $A$, the \emph{principal characteristic permanent rank sequence} of $A$ (abbreviated ppr-sequence of $A$ or $\ppr(A)$) is defined as $r_0 r_1 r_2 \ldots r_n$ where for $1 \leq k \leq n$
%\[ r_k= \begin{cases}
%1 & \text{if $A$ has a principal submatrix of size $k$ with nonzero permanent, and }\\
%0 & \text{otherwise,}
%\end{cases}\]
%while $r_0 = 1$ if and only if $A$ has a zero on its main diagonal.
%\end{defi}

%For a (directed) graph $G$ on $n$ vertices $1$, $2$, $\ldots$, $n$, and $\alpha \subseteq [n]$, the graph $G[\alpha]$ is the induced subgraph of $G$ on vertices in $\alpha$. For an $n\times n$ matrix $A$ and $\alpha \subseteq [n]$, $A[\alpha]$ denotes the principal submatrix of $A$ from rows and column indexed by $\alpha$. The zero-nonzero pattern of $A$ is a $(0,1)$-matrix $B$ of the same order where $B_{ij}=1$ if and only if $A_{ij}\neq 0$. The following proposition shows that the ppr-sequences of a nonnegative matrix and its zero-nonzero pattern are equal.

%%FK NOTE: We decided to define a generalized cycle with permutations. The above definition is much more succinct.

Further, a generalized cycle of order $|G|$ is said to be \emph{spanning}. Next, recall that a \emph{matching} is a collection of disjoint edges. Since a matching in an undirected graph can be viewed as a disjoint collection of directed 2-cycles, every matching forms a generalized cycle.  %{\color{blue} Why must the graph be undirected?} 
%Answer: Because a matching on a directed graph does not necessarily make sense.
The set of all generalized cycles of order $k$ of a (directed) graph $G$ is denoted by $\cyc_k(G)$.

%With a generalized cycle $C$ of a (directed) graph $G$, we associate a permutation of the vertices of $C$ as follows. Each (directed) cycle $(v_1, v_2, ..., v_k)$ is associated with the cyclic permutation $(v_1 v_2 \ldots v_k)$ and each loop $v_1 v_1$ is associated with the permutation $(v_1)$ which fixes $v_1$. The permutation $\pi_C$ is defined to be the product of these associated permutation cycles.
%\end{defi}

The connection between generalized cycles and permanent ranks is made formal by the following proposition.

\begin{prop}\label{gencycle}
	Let $G$ be the underlying (directed) graph of the nonnegative matrix $A$ and let $\ppr(A) = r_0 r_1  \ldots r_{n}$. For $k \ge 1$, $r_k = 1$ if and only if $G$ has a (directed) generalized cycle of order $k$. 
\end{prop}
\begin{proof}
	Let $\alpha  \subseteq [n]$ with $|\alpha| = k$. %First, note that $\per(A[\alpha]) \geq 0$. 
	\begin{align*}
		\per(A[\alpha]) &= \sum_{\pi \in S_\alpha}\prod_{j = 1}^k a_{i_j,\pi(i_j)}
	\end{align*}
	where $\alpha = \{i_1, ..., i_k\}$. A term of the sum above is nonzero (and positive) if and only if $\pi \in \cyc_k(G)$

%Thus, $\per(A[\alpha]) > 0$ if and only if $G[\alpha]$ has a spanning generalized cycle. That is $r_k = 1$ if and only if $G$ has a generalized cycle of order $k$, for $k \geq 1$.
\end{proof}

%{\color{blue} I think we're being a little ambiguous about whether the graph/generalized cycle is directed or not. I suggest defining the generalized cycle with directed edges and then viewing all graphs as directed.}
% FK answer: We define a generalized cycle as a permutation which implicitly are directed edges.

We say a binary sequence $r_0 r_1 \ldots r_n$ is \emph{realizable}, if there is a matrix whose ppr-sequence is $r_0 r_1 \ldots r_n$.
%
%In Section \ref{sectiondirected} we characterize all principal permanent rank characteristic sequences of (entry-wise) nonnegative matrices. In Section \ref{sectionundericted} we  characterize all principal permanent rank characteristic sequences of symmetric (entry-wise) nonnegative matrices. \comment{In Section \ref{sectionotherresults} we  provide some other results (shall we omit this? and this) In section \ref{sectionfuture} some questions regarding the principal permanent rank characteristic sequences of matrices are asked.}

\section{General Nonnegative Matrices}\label{sectiondirected}
In this section, we characterize the principal permanent rank sequences of nonnegative matrices. We prove:

\begin{thm}\label{zerononsymmetric}\label{onenonsymmetric}
The binary sequence $r_0 r_1 \ldots r_n$ is realizable as a ppr-sequence of a nonnegative matrix if and only if
\begin{itemize}
\item $r_0 = 0$ and $r_i = 1$ for $i=1,2,\ldots, n$ or
\item $r_0 = 1$.
\end{itemize}
%then there is a nonnegative $n \times n$ matrix $A$ such that $\ppr(A) = r_0 r_1 \ldots r_n$. 
\end{thm}

First, let us prove the following lemma.

\begin{lem}\label{lem:zeroout}
Let $A$ be a nonnegative $n \times n$ matrix with $\ppr$-sequence $r_0 r_1 \ldots r_n$. If $r_0 = 0$, then $r_i = 1$ for all $i=1, 2, \ldots, n$.
\end{lem}

\begin{proof}
Recall $r_0 = 0$ if and only if there is a loop on every vertex in  {the underlying} graph $G$. Thus, for all $k \in [n]$, $G$ has a generalized cycle of order $k$ consisting of $k$ loops. Therefore, by Proposition \ref{gencycle}, $r_k = 1$ for all $k \in [n]$. Lastly, any sequence of the form $r_0 r_1 \ldots r_n = 0 1 1 \ldots 1$ is realized by $I_n$, the identity matrix of order $n$.
\end{proof}

\begin{proof}[Proof of Theorem \ref{zerononsymmetric}]
	The case when $r_0 = 0$ is covered by Lemma \ref{lem:zeroout}. Hence, we can assume $r_0=1$.
%\end{proof}
%\begin{thm}\label{onenonsymmetric}
%	Any sequence of the form $1 r_1 r_2 \ldots r_n$ is realizable by the adjacency matrix of a directed graph, where $n \geq 2$.
%\end{thm}
%\begin{proof}
We will construct a directed graph, $G$, as follows. Start with the directed path $v_1 \rightarrow v_2 \rightarrow \cdots \rightarrow v_n$. Next, for each $k \in [n]$, add a directed edge from $v_k$ to $v_1$ if and only if $r_k = 1$ (see Figure \ref{fig.backtrack}).

	Let $A$ be the adjacency matrix of $G$ and $\ppr(A) = q_0 q_1 \ldots q_n$. We claim that $q_i = r_i$ for each $i \in [n]$.  First note that $r_0 = 1$, because $a_{22} = 0$.  Now consider $r_k$, for $k \geq 1$.  If $r_k = 1$, then there is an edge from $v_k$ to $v_1$. Hence $C = (v_1, ..., v_k)$ is a directed generalized cycle of order $k$ in $G$. Thus, by Proposition \ref{gencycle}, $q_k = 1$. 
	
	Now suppose that $r_k = 0$ and consider a subset $S$ of $k$ vertices. If $v_1 \notin S$, then $G[S]$ is a disjoint union of directed paths and thus has no spanning generalized cycle. Now assume that $v_1 \in S$. If $v_j \in S$ for some $j >k$, then $v_i \notin S$ for some $1 < i \leq k$. Thus, $G[S]$ has no generalized cycle containing $v_j$. Finally, if $S = \{v_1, v_2, ..., v_k\}$, $G[S]$ is a graph on $k$ vertices with a pendent vertex $v_k$. Therefore, $G[S]$ has no spanning generalized cycle. Hence, by Proposition \ref{gencycle}, $q_k = 0$. %\comment{Could someone check this proof please? there were a few minor errors that I think they're fixed now.}\commentn{Seems ok to me (--Josh)}.
	%FK NOTE: Commented out comments.
\end{proof}

\begin{figure}[t]
\begin{center}
\begin{tikzpicture}
	\foreach \i in {1,2,...,10}
		\node[draw, circle, fill, scale =.5] (\i) at (\i,0) {};
	\foreach \from/\to in {1/2,2/3,3/4,4/5,5/6,6/7,7/8,8/9,9/10}
		\draw[-latex] (\from) to[out=-30,in=210] (\to);
	\foreach \from in {3,5,8}
		\draw[-latex] (\from) to[out=150,in=30] (1);
	\node[] () at (1,-.5) {$v_1$};	
	\node[] () at (5,-.5) {$v_k$};
	\node[] () at (10,-.5) {$v_n$};
	% If you don't want the dots, comment out the next four lines
	\node[draw,fill,white, scale=.6] () at (3.45,-.2) {$\cdots$};
	\node[scale = .7] () at (3.45,-.2) {$\cdots$};
	\node[draw,fill,white, scale=.6] () at (6.45,-.2) {$\cdots$};
	\node[scale = .7] () at (6.45,-.2) {$\cdots$};
\end{tikzpicture}
\caption{An illustration of the construction in Theorem \ref{onenonsymmetric}.}\label{fig.backtrack}
\end{center}
\end{figure}
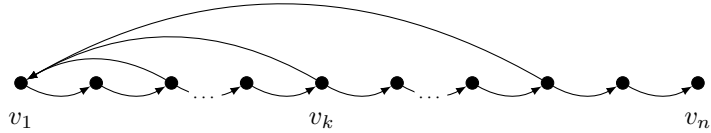

\section{Nonnegative Symmetric Matrices}\label{sectionundericted}
In this section, we consider the principal permanent rank characteristic sequences of nonnegative symmetric matrices. In contrast to general nonnegative matrices, the set of allowable sequences is more restrictive. The key difference between the symmetric and general case is that in the symmetric case a single graph edge always counts as a $2$-cycle. For example, $r_2 = 1$ if the underlying graph {has} an edge. Moreover, since every even cycle contains a perfect matching, this implies that we may always choose a generalized cycle where all even cycles are $2$-cycles. {That is, if $G$ contains a generalized cycle of order $k$, then one realization consists of a (possibility empty) matching and a (possibility empty) collection of odd cycles.}

%That is, if there is a generalized cycle of order $k$, then there is a disjoint union of a matching and some odd cycles of total order $k$.

Ultimately,  in Theorem \ref{fuzzyegg}, we fully characterize which $\ppr$-sequences are realizable by nonnegative symmetric matrices. First, we provide some necessary conditions for a binary sequence to be realizable in Lemmas \ref{lem:zeroout}--\ref{lem:minmaxodd}. 

%\begin{remark} 
%\end{remark}
% FK NOTE: Removed remark and moved it into its own paragraph.

The following lemma shows if there is no generalized cycle of an even order $2k$, then every generalized cycle of the graph is smaller than $2k$.

\begin{lem}\label{lem:even}
	Suppose $A$ is a nonnegative $n \times n$ symmetric matrix, and let $ppr(A) = r_0 r_1 r_2 \ldots r_{n}$.  If $r_{2k} = 0$ for some $k > 0$, then $r_{j} = 0$ for all $j \geq 2k$.  
\end{lem}
\begin{proof}
	Recall that Proposition \ref{gencycle} asserts that $r_{j} = 1$ if and only if the underlying graph, $G$, has a generalized cycle on $j$ vertices. 
	First, suppose $r_{j} = 1$ for some odd $j = 2t+1$. Then there exists a generalized cycle of $G$ consisting of at least one odd cycle, along with (possibly) a matching. Discarding an arbitrary vertex from this odd cycle results in a path on an even number of vertices. This path contains a spanning matching.  When this matching is considered with the other components of the original odd generalized cycle, we have a generalized cycle on $j-1$ vertices. Thus, $r_{j-1} = 
	%r_{2t} = 
	1$.

	Now, suppose that $r_{j}=1$ for some $j=2t$.  Then there is a generalized cycle $C$ of order $j$ consisting of a (possibly empty) collection of odd cycles plus a (possibly empty) matching.  If $C$ contains a matching edge, then discarding it yields a generalized cycle on $2t-2$ vertices. Hence, $r_{j-2} = r_{2t-2} = 1$.  Otherwise, $C$ consists solely of an even number of odd cycles.  Discarding one vertex each from two different odd cycles, and noting again that the remaining even paths contain a spanning matching, yields a generalized cycle on $2t-2$ vertices. That is, $r_{2t-2} = 1$.  

	Therefore if $r_{2t+2} = 1$, or $r_{2t+1} = 1$, we have that $r_{2t} = 1$, and this implies the lemma.
%	\comment{someone please proofread this.}\commentn{Seems ok to me (--Josh)}
%FK NOTE: commented out comment
\end{proof}

{The proof of Lemma \ref{lem:even}  further demonstrates that if an even generalized cycle exists, then there is a generalized cycle for all smaller even orders. Thus, we are left to study the restriction that odd cycles impose on the $\ppr$-sequence. In Lemma \ref{lem:oddspan}, we show that the odd indices $i$ for which $r_i = 1$ must be sequential}; however, first we make a few structural observations.

\begin{fact} \label{lem.oddgirth}
Let $2\ell + 1$ be the length of the shortest odd cycle %the odd girth 
of $G$, then $2 \ell + 1$ is the smallest odd integer $k$ with $r_{k} = 1$.
\end{fact}
%
%\begin{proof}
%Since any 
%\end{proof}

\begin{lem}\label{lem:struct}
	Suppose $r_{2k-1} = 0$ and $r_{2k+1} = 1$. Then every generalized cycle on $2k+1$ vertices is a $2k+1$ cycle, and the vertex set of that generalized cycle induces a cycle with no chords.  
\end{lem}
\begin{proof}
	Consider a generalized cycle on $2k+1$ vertices. {As noted before, we can choose a generalized cycle consisting} of a collection of odd cycles plus a matching. If there is an edge in the matching, however, discarding it yields a generalized cycle on $2k-1$ vertices, that is, $r_{2k-1} = 1$. Thus, the generalized cycle is a collection of odd cycles. {If there is more than one odd cycle in the collection}, one vertex {can be} discarded from two different cycles, and a perfect matching can be taken from the resulting even paths to find a generalized cycle on $2k-1$ vertices.  Thus, assuming $r_{2k-1}=0$, the generalized cycle is a single cycle. 
	
	{Now let $V$ be the vertex set for some generalized cycle of order $2k+1$}. The set $V$ induces a $2k+1$ cycle, along with potentially some chords.  If there is a chord, however, {$G[V]$} also consists of a smaller odd cycle (containing the chord) plus an even path {on the remaining vertices} containing at least one edge.  Converting this path to a matching and discarding an edge would yield a generalized cycle on $2k-1$ vertices, completing the proof of the lemma.
\end{proof}

\begin{lem}\label{lem:oddspan}
	Suppose $r_{2i+1} = r_{2k+1} = 1$ for some integers $i < k$, then $r_{t} = 1$ for all $2i+1 \leq t \leq 2k+1$.  
\end{lem}
\begin{proof}
	By Lemma \ref{lem:even} it suffices to just consider $r_t$ for odd $t$. %\comment{Isn't the rest of the proof just a result of the previous lemma? I mean, if there is a zero in between pick the largest index and by previous lemma everything smaller than that would be zero, contradicting the assumption that $r_{2i+1} =1$. I haven't read the following proof recently and I might be quite wrong. But we might be proving something more than what's in the statement?}\commentn{I don't think it follows immedaitely from the previous lemma, but maybe it does?  I don't think this proof gives anything more than what was stated.(--Josh)}\\
	%%FK NOTE: Commented comments
	
	It also suffices to show that $r_{2k-1} = 1$.  Suppose to the contrary that $r_{2k-1} = 0$.  By Lemma \ref{lem:struct}, every generalized cycle  of size $2k+1$ is an induced cycle.  Fix such a generalized cycle on vertex set $V$.  Suppose $j < k$ is minimum with the property that $r_{2j-1} = 1$.  Again, fix a generalized cycle with size $2j-1$. This is also an (induced) cycle on a vertex set $V'$. 

	If $V' \cap V = \emptyset$, then we are done: discarding a vertex from $V$ we have a path on $2k$ vertices, and a cycle on $2j - 1$ vertices.  This path can be treated as a matching, and discarding sufficiently many edges in the matching yields a generalized cycle of size ${2k-1}$.  Otherwise, we may assume that the cycle on $V'$ follows along the cycle on $V$ on $s$ contiguous segments {sharing a total of} $\ell$ vertices.  Immediately following each segment there must be at least one vertex in $V' \setminus V$, so $s + \ell \leq 2j-1$.  The vertices in $V$ {\it not} in $V'$ lie on $s$ segments with total length $2k+1 - \ell.$  For parity reasons, we may have to delete one vertex from each segment, but we can then obtain a matching on {at least}  $$2k+1-\ell - s$$ vertices.  Combining this matching with the cycle on $2j-1$ vertices,  we have a generalized cycle on $$2k+1 + (2j-1) - \ell - s \geq 2k + 1$$ vertices comprised of a cycle on $2j-1$ vertices plus a matching.  Discarding sufficiently many edges of the matching, we again obtain a generalized cycle of size $2k-1$.
\end{proof}

Finally, the following lemma shows that the largest odd generalized cycle of a graph strongly constrains the largest even generalized cycle.

\begin{lem}\label{lem:minmaxodd}
	Suppose $m$ and $M$ are respectively the smallest and largest odd integers so that $r_m = r_{M} = 1$. Assuming $m+M+2 \le n$ then $r_{m+M+2} = 0$.
\end{lem}
\begin{proof}
	Let $t$ be the largest even number so that $r_t = 1$, and consider a generalized cycle $C_1$ on $t$ vertices.  We claim that if $t > M+1$ then there is a matching on $t$ vertices. Indeed, if the generalized cycle on $t$ vertices contains an odd cycle %\comment{Is this needed?--and hence, by parity considerations, at least two odd cycles--} 
	%% FK NOTE: commented comment
	a single vertex can be discarded from one odd cycle to obtain a generalized cycle on $t-1$ vertices, and hence $t-1 \leq M$.  (Note that it is possible that $t-1 < M$, if the largest generalized cycle in the graph had odd order.)

	Thus, we may assume that the generalized cycle $C_1$ consists of $\frac{t}{2}$ disjoint edges.  Consider a generalized cycle $C_2$ on $m$ vertices, which may include vertices from at most $m$ of the disjoint edges of $C_1$. Taking $C_2$ along with the edges of $C_1$ not containing a vertex from it yields an odd generalized cycle on at least $m + 2(\frac{t}{2} - m) = t - m$ vertices.  Therefore $M \geq t-m$. Rearranging, we get $t \leq M + m$, which proves the lemma.  
\end{proof}

The following Theorem shows that the above necessary conditions on the ppr-sequence of a nonnegative symmetric matrix are indeed sufficient. That is, if a binary sequence $r_0 r_1 \ldots r_n$ satisfies the conditions of Lemmas \ref{lem:even}--\ref{lem:minmaxodd}, then there is a nonnegative symmetric matrix whose ppr-sequence is $r_0 r_1 \ldots r_n$.

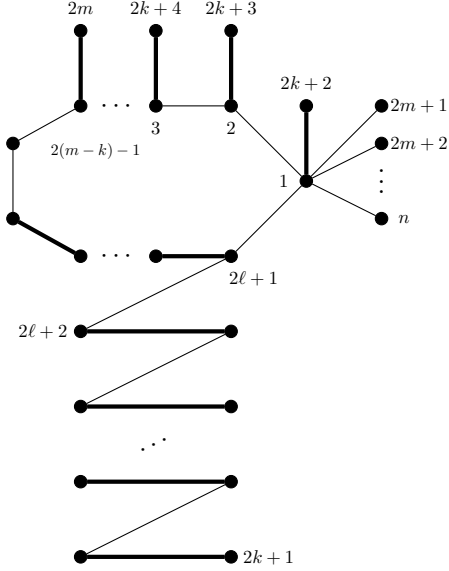
\begin{figure}[h]
\begin{center}
\begin{tikzpicture}
	\node[draw, circle, fill, scale =.5] (1) at (7,0) {};
	\node[scale =.7] () at (6.7,0) {$1$};
	
	\node[draw, circle, fill, scale =.5] (1a) at (8,1) {};
	\node[scale =.7] () at (8.5,1) {$2m+1$};
	\node[draw, circle, fill, scale =.5] (1b) at (8,.5) {};
	\node[scale =.7] () at (8.5,.5) {$2m+2$};
	\node[] () at (8,.1) {$\vdots$};
	\node[draw, circle, fill, scale =.5] (1c) at (8,-.5) {};
	\node[scale =.7] () at (8.3,-.5) {$n$};
	
	\node[draw, circle, fill, scale =.5] (2) at (6,1) {};
	\node[scale =.7] () at (6,.7) {$2$};
	
	\node[draw, circle, fill, scale =.5] (3) at (5,1) {};
	\node[scale =.7] () at (5,.7) {$3$};
	
	\node[draw, circle, fill, scale =.5] (4) at (4,1) {};
	\node[scale =.6] () at (4.2,.4) {$2(m-k)-1$};

	\node[draw, circle, fill, scale =.5] (5) at (3.1,.5) {};
	
	\node[draw, circle, fill, scale =.5] (5a) at (3.1,-.5) {};
	
	\node[draw, circle, fill, scale =.5] (6) at (4,-1) {};
	
	\node[draw, circle, fill, scale =.5] (7) at (5,-1) {};
	
	\node[draw, circle, fill, scale =.5] (8) at (6,-1) {};
	\node[scale =.7] () at (6.3,-1.3) {$2\ell+1$};
	
	\node[draw, circle, fill, scale =.5] (111) at (7,1) {};
	\node[scale =.7] () at (7,1.3) {$2k+2$};
	\node[draw, circle, fill, scale =.5] (222) at (6,2) {};
	\node[scale =.7] () at (6,2.3) {$2k+3$};	
	\node[draw, circle, fill, scale =.5] (333) at (5,2) {};
	\node[scale =.7] () at (5,2.3) {$2k+4$};
	\node[draw, circle, fill, scale =.5] (444) at (4,2) {};
	\node[scale =.7] () at (4,2.3) {$2m$};
	
	\draw[] (1a)--(1)--(1c);
	\draw[] (1)--(1b);
	\draw[] (111)--(1)--(2)--(3)--(333);
	\draw[] (2)--(222);
	\draw[] (444)--(4)--(5)--(5a)--(6);
	\node[] () at (4.5,1) {$\cdots$};
	\draw[] (7)--(8)--(1);
	\node[] () at (4.5,-1) {$\cdots$};
	
	\node[draw, circle, fill, scale =.5] (9) at (4,-2) {};
	\node[scale =.7] () at (3.5,-2) {$2\ell+2$};
	\node[draw, circle, fill, scale =.5] (10) at (6,-2) {};
	\node[draw, circle, fill, scale =.5] (11) at (4,-3) {};
	\node[draw, circle, fill, scale =.5] (12) at (6,-3) {};
	
	\node[draw, circle, fill, scale =.5] (13) at (4,-4) {};
	\node[draw, circle, fill, scale =.5] (14) at (6,-4) {};
	\node[draw, circle, fill, scale =.5] (15) at (4,-5) {};
	\node[draw, circle, fill, scale =.5] (16) at (6,-5) {};
	\node[scale =.7] () at (6.5,-5) {$2k+1$};
	
	\draw[] (8)--(9)--(10)--(11)--(12);
	\node[rotate=25] () at (5,-3.5) {$\cdots$};	
	\draw[] (13)--(14)--(15)--(16);

	\draw[ultra thick] (1)--(111);
	\draw[ultra thick] (2)--(222);
	\draw[ultra thick] (3)--(333);
	\draw[ultra thick] (4)--(444);
	\draw[ultra thick] (5a)--(6);	
	\draw[ultra thick] (7)--(8);	
	\draw[ultra thick] (9)--(10);	
	\draw[ultra thick] (11)--(12);	
	\draw[ultra thick] (13)--(14);	
	\draw[ultra thick] (15)--(16);	
	
%	\def \n {10}
%	\def \radius {2.5cm}
%	\def \margin {8} % margin in angles, depends on the radius
%
%	\foreach \s in {1,...,\n}
%	{
%		\node[draw, circle, fill, scale =.5] () at ({360/\n * (\s - 1)}:\radius/3*2) {};
%		\draw[-] ({360/\n * (\s - 1)}:\radius/3*2) -- ({360/\n * (\s)}:\radius/3*2);
%		
%		\node[draw, circle, fill, scale =.5] () at ({360/\n * (\s - 1)}:\radius) {};
%		\draw[-] ({360/\n * (\s)}:\radius/3*2) -- ({360/\n * (\s)}:\radius);
%	}
%	
%	\draw[-] ({360/\n * (2 - 1)}:\radius/3*2) -- ({360/\n * (4)}:\radius/3*2);
%	
%	\node[scale = .7] () at ({360/\n * (2 - 1)}:\radius/2) {$v_1$};
%	\node[scale = .7] () at ({360/\n * (3 - 1)}:\radius/2) {$v_2$};
%	
%	\node[scale = .7] () at ({360/\n * (5 - 1)}:\radius/2) {$v_{2\ell +1}$};
%	
%	\node[scale = .7] () at ({360/\n * (10 - 1)}:\radius/2) {$v_{2k}$};
%	\node[scale = .7] () at ({360/\n * (1 - 1)}:\radius/2) {$v_{2k+1}$};
\end{tikzpicture}
\caption{An illustration of the construction of Case 4 in the proof of Theorem \ref{fuzzyegg}.} \label{fig.fuzzyegg}%\comment{do we need all the \textit{fuzz}es in the picture or some?}
\end{center}
\end{figure}

\begin{thm}
%[Fuzzy Egg Theorem] 
%% Okay, it can't be called the fuzzy Egg Theorem, but I am not changing all of the labels. FK - 7/15
\label{fuzzyegg}
Any binary sequence not discounted by Lemmas \ref{lem:zeroout}--\ref{lem:minmaxodd} is realizable by a symmetric nonnegative matrix.

That is, $r_0 r_1 \ldots r_n$ is realizable as a ppr-sequence of a nonnegative symmetric matrix if and only if
\begin{itemize}
\item $r_0 = 0$ and $r_i = 1$ for $i=1,2,\ldots, n$; or
\item there are nonnegative integers $\ell, k, m$ with $\ell \le k \leq m \le \ell + k + 1$ where
	\begin{enumerate}[a)]
		\item $r_{2j+1} = 0$ for any $j < \ell$,
		\item $r_{2j+1} = 1$ for any $j$ with $ \ell \le j \le k$,
		\item $r_{2j+1} = 0$ for any $j$ with $k < j  \leq \frac{n-1}{2}$,%\comment{ $\leq \frac{n-1}{2}$ Shall we omit this?}$,
		\item $r_{2j} = 1$ for any $0 \leq j \leq m$, and
		\item $r_{2j} = 0$ for any or $m < j  \leq \frac{n-1}{2}; or$ %comment{ $\leq \frac{n}{2}$ Shall we omit this?}$,
	\end{enumerate}
\item $r_0 = 1$,  $r_{i} = 0$ for all odd $i \le n$, $r_{i} = 1$ for all even $i \le 2m$ and $r_{i} = 0$ for all  even $i > 2m$ for some nonnegative $m \le \lfloor \frac{n-1}{2} \rfloor.$\end{itemize}
\end{thm}

%{\color{blue} Should the last bullet be $r_0 = 1$ since $r_0 = 0 \rightarrow r_i = 1 \hspace{2mm} \forall i \geq 1$?}
\begin{proof}
First note that Lemma \ref{lem:minmaxodd} implies $m \leq k + \ell + 1 $. Also, Fact \ref{lem.oddgirth} implies that for any odd $t$ where $t$ is less than the length of the shortest odd cycle of the graph, then $r_t = 0$. Hence it implies item \textit{a}. Lemma \ref{lem:oddspan} implies for any odd $t$ between the length of the shortest cycle and the length of the largest generalized cycle of the graph, $r_t = 1$. That shows the necessity of item \textit{b}. Lemma \ref{lem:zeroout} asserts $r_0 = 1$, and Lemma \ref{lem:even} implies for even numbers $t$ no more than a fixed number, $r_t = 1$, and for even numbers $t$ more than that fixed number, $r_t = 0$. This implies items \textit{d} and \textit{e}. Now, since Lemmas \ref{lem:zeroout}--\ref{lem:minmaxodd} show the necessity of the conditions above, it suffices to construct a matrix for various cases. 

 \vspace*{1mm}
 \noindent \textbf{Case 1}: $r_0 = 0$ and $r_i = 1$ for $i=1,2,\ldots, n$.

 \noindent This case is covered by Lemma \ref{lem:zeroout} where the identity matrix, $I_n$, realizes the sequence.

 \vspace*{1mm}
 \noindent \textbf{Case 2:} $0 = \ell = k = m$ (i.e., $r_0 = r_1 = 1$ and $r_i = 0$ for $i = 2, 3, \ldots n$).

 \noindent Consider the graph with $n$ isolated vertices where one vertex has a loop. The adjacency matrix has $r_0 =r_ 1= 1$ and $r_i = 0$ otherwise.

 \vspace*{1mm}
 \noindent \textbf{Case 3:} $r_0 = 1$ and $0 < \ell = k = m$.

 \noindent Consider a cycle on $2 \ell + 1$ vertices and $n- 2\ell - 1$ isolated vertices. The only odd generalized cycle is on $2\ell +1$ vertices, and there is a matching on the cycle for all even $2j$ for $j \le \ell$.

 %\noindent \textbf{Case 4:} $r_0 = 1$ and $0 = \ell$ and $\ell, k, m$ are not all equal.

 \vspace*{1mm}
 \noindent \textbf{Case 4:} $r_0 = 1$ and $\ell, k, m$ are not all equal.

 \noindent We construct a graph as follows. Construct an odd cycle on vertices $1,2,\ldots, 2\ell+1$ (if $\ell = 0$, take the odd cycle to be a loop on a single vertex), and a path on vertices $2\ell+1, 2\ell+2, \ldots, 2k+1$. Add $2(m-k)-1$ vertices and connect each of them to one of the vertices $1,2,\ldots, 2(m-k)-1$ such that no pair of them is connected to the same vertex. Finally, add $n-2m$ vertices and connect all of them to vertex 1. See Figure \ref{fig.fuzzyegg}.

 We now verify that items \textit{a}--\textit{e} hold. Items \textit{a}--\textit{c} assert that the smallest and the largest generalized cycles of the graph are to be of sizes $2\ell+1$ and $2k+1$, respectively. Also, \textit{d} and \textit{e} assert that the graph has to have a maximum matching of size $2m$. 

The smallest odd cycle of $G$ is of size $2\ell+1$, hence 

\begin{itemize}
	\item[\textit{a)}] $r_{2j+1} = 0$ for $j<\ell$, and
\end{itemize}

Now, consider the $2\ell+1$ cycle joint with (possibly zero) disjoint edges from the path. This shows there are generalized cycles of length $2j+1$ for $\ell \leq j \leq k$. That is
\begin{itemize}
	\item[\textit{b)}] $r_{2j+1} = 1$ for $2j+1$ for $\ell \leq j \leq k$, and 
	\item[\textit{c)}] $r_{2j+1} = 0$ for any $j$ with $k < j$.
\end{itemize}

Note that the graph has a maximum matching of size $m$. This is obtained by taking the edges that connect each of the vertices $1,2,\ldots, 2(m-k) - 1$ to the pendent vertex adjacent to them ($2(m-k)-1$ edges), every other edge in { the }rest of the $2\ell+1$ cycle ($\frac{2\ell+1 - 2(m-k)-1}{2}$ edges), and the maximum matching from the path ($k-\ell$ edges). %\comment{Yes! It adds up.} 
Thus,
\begin{itemize}
	\item[\textit{d)}] $r_{2j} = 1$ for any $1 \leq j \leq m$, and
	\item[\textit{e)}] $r_{2j} = 0$ for any $m < j$.
\end{itemize}

 \vspace*{1mm}
 \noindent \textbf{Case 5:} $r_0 = 1$,  $r_{i} = 0$ for all odd $i \le n$, $r_{i} = 1$ for all even $i \le 2m$ and $r_{i} = 0$ for all  even $i > 2m$ for some nonnegative $m \le \lfloor \frac{n-1}{2} \rfloor$.
 
% {\color{blue} Here too, should this be $r_0 = 1$?}
 
\noindent This case is obtained with a graph with $m$ disjoint edges and $n-2m$ isolated vertices.

\end{proof}

%Observe that the constructed graph is connected when $k$, $\ell$, and $m$ are not all equal. %\comment{It is easy to see that such a graph cannot be connected when $k = \ell = m < \lfloor n/2 \rfloor$. Can someone check this please?}
%When $k=l=m=\lfloor \frac{n}{2} \rfloor$ the graph becomes an odd cycle. Also, when $k=l=m < \lfloor \frac{n}{2} \rfloor$ the graph consists of a cycle and some disjoint vertices.
%The lemmas in this section demonstrate that the conditions are necessary. Hence, given $\ell$, $k$, and $m$, it suffices to construct a graph with the corresponding $ppr(A)$ as follows.

%First, we note that it suffices to consider the case where $r_n = 1$. If $q$ is the largest integer such that $r_q=1$, then add $n-q$ isolated vertices, and consider the remaining $q$ vertices. For the remainder we will assume that $r_n=1$.
%
%\emph{Case 1: $n$ is odd}
%
%If $n$ is odd, then $2k+1 = 1$ 

%Let $2k+1$ be the largest odd number such that $r_{2k+1} = 1$. Let $2 \ell + 1$ be the smallest odd number such that $r_{2\ell +1}=1$. And let $2m$ be the largest even number such that $r_{2m} = 1$. Then, consider a cycle of size $2k+1$ with a single chord creating a cycle of length $2 \ell +1$. Finally, add $m-\ell$ vertices each with degree 1 whose edges are incident to a \comment{\ldots}(See Figure \ref{fig.fuzzyegg}).
%
%To show that these realize a specific permissible sequence, \comment{\ldots}
%
%By Lemma \ref{lem.oddgirth} \comment{\ldots}

\section{Skew-symmetric Matrices}\label{skew}

Previously, we only considered nonnegative matrices. This consideration benefited the analysis as every contribution to the permanent was necessarily positive. 

In this section, we consider \emph{skew-symmetric} matrices. Recall that a real matrix $A$ is \emph{skew-symmetric} if $A_{ji} = -A_{ij}$. First note that the odd positions in the \ppr-sequence of a skew-symmetric matrix have to be all zero, as shown in the following lemma.

\begin{lem}\label{lem.skewnoodd}
Let $A$ be a skew-symmetric matrix with $\ppr(A) =r_0 r_1 \cdots r_n$. Then $r_{2i+1} = 0$ for all integers $i$ with $0 \leq i \leq \lfloor \frac{n-1}{2} \rfloor$.
\end{lem}

\begin{proof}
%Let $B$ be a principal submatrix of $A$ of size $k$. Let $\sigma$ be any permutation in $S_k$ with inverse $sigma^{-1}$. Since, $b_{ii} = 0$ for all $i$, any $\sigma$ with a fixed-point will add 0 toward the permanent. Additionally, since $2i+1$ is odd, no $\sigma$ without a fixed point is its own inverse. Therefore, it suffices to show that:
%\[  \prod_{i=1}^{k} b_{i\,\sigma(i)} = -prod_{i=1}^{k} b_{i\,\sigma^{-1}(i)}. \]

Choose $k \le n$ odd. Let $B = (b_{ij})$ be a principal submatrix of $A$ of size $k$. We will show that $\per(B) = 0$.
\begin{eqnarray*}
\per(B) &=& \sum_{\sigma\in S_k} \left( \prod_{i=1}^{k} b_{i\,\sigma(i)}\right) \\
&=&  \sum_{\sigma\in D_k} \left( \prod_{i=1}^{k} b_{i\,\sigma(i)}\right) + \sum_{\sigma\in S_k\setminus D_k} \left( \prod_{i=1}^{k} b_{i\,\sigma(i)}\right) 
\end{eqnarray*}
where $D_k$ is the set of all derangements on $[k]$ (i.e., permutations without a fixed point). For $\sigma \in S_k \setminus D_k$, $i = \sigma_i$ for some $i$, and so $b_{i\,\sigma(i)}=0$ by skew-symmetry. Further, since $k$ is odd, no $\sigma \in D_k$ is its own inverse. Therefore, continuing from above, we have 
\begin{eqnarray*}
\per(B) &=& \sum_{\sigma\in D_k} \left( \prod_{i=1}^{k} b_{i\,\sigma(i)}\right) \\
&=& \sum_{\sigma\in D'_k} \left( \prod_{i=1}^{k} b_{i\,\sigma(i)}+  \prod_{i=1}^{k} b_{i\,\sigma^{-1}(i)}\right)
\end{eqnarray*}
where $D'_k \subset D_k$ is a maximum subset of derangements where no elements is an inverse of another. However, by skew-symmetry,
\begin{eqnarray*}
&=& \sum_{\sigma\in D'_k} \left( \prod_{i=1}^{k} b_{i\,\sigma(i)}+ \prod_{i=1}^{k} -b_{\sigma^{-1}(i)\,i}\right) \\
&=& \sum_{\sigma\in D'_k} \left( \prod_{i=1}^{k} b_{i\,\sigma(i)}+\prod_{i=1}^{k} -b_{i\, \sigma(i)}\right) \\
&=& \sum_{\sigma\in D'_k} \left( \prod_{i=1}^{k} b_{i\,\sigma(i)} +  (-1)^k \prod_{i=1}^{k} b_{i\, \sigma(i)}\right) \\
&=& \sum_{\sigma\in D'_k} \left( \prod_{i=1}^{k} b_{i\,\sigma(i)} - \prod_{i=1}^{k} b_{i\, \sigma(i)}\right) \\
&=& 0
\end{eqnarray*}

%Let $\alpha \subseteq [n]$ with $|\alpha| = 2i+1$. Let $\alphb_1, \alphb_2, \ldots, \alphb_{2i+1}$ be an ordering of alpha such that $A(\alphb_j,\alpha{j+1} is a directed cycle cover of $G[\alpha]$ then it must contain an odd cycle $C_1$. Without loss of generality suppose $\mc{C}$ contributes $1$ to $\per(A[\alpha])$. Then reversing the edges in $C_1$, we obtain another cycle cover which contributes $-1$ to $\per(A[\alpha])$.
\end{proof}
  
%\begin{lem}
%Let $A$ be a skew-symmetric matrix with a nonzero entry and  perrank sequence $\per(A) =q_0, q_1, \cdots, q_n$. Then , $q_2 = 1$. \end{lem}  
Now, the question is to characterize the patterns of zeros and ones in the even positions of this sequences. Concentrating on the even positions, several examples of small size are checked and it is observed that there are no gaps between the ones in the even positions. It is easy to see that this property holds for trees. In the following theorem we will characterize $\ppr(A)$ for all skew-symmetric matrices whose underlying graph is a tree.

\begin{thm}\label{thm.skewtree}
Let $A$ be a skew-symmetric matrix whose underlying graph is a tree with a maximum matching of size $\mu(G)$. Then, the principal permanent rank sequence $\ppr(A) = r_0 r_1 \cdots r_n$ has $r_{k} = 1$ if and only if $k$ is even and $k \le 2\mu(G)$. Furthermore, any such sequence is realizable by a skew-symmetric matrix whose underlying graph is a tree.\end{thm}

\begin{proof}
If $k\le n$ is odd, then $q_k = 0$ by Lemma \ref{lem.skewnoodd}. Choose $k$ even and $\alpha \subset [n]$ with $|\alpha|=k$ Let $B = A [\alpha]= (b_{i,j})$. We will show the permanent of $B$ is nonzero if $k \le 2 \mu(G)$ and $0$ otherwise.

\begin{eqnarray*}
\per(B) &=& \sum_{\sigma\in S_k} \left( \prod_{i=1}^{k} b_{i\,\sigma(i)}\right)\\
&=& \sum_{\sigma\in M_{k/2}} \left( \prod_{i=1}^{k} b_{i\,\sigma(i)}\right)+ \sum_{\sigma\in D_k \setminus M_{k/2}} \left( \prod_{i=1}^{k} b_{i\,\sigma(i)}\right) + \sum_{\sigma\in S_k\setminus (D_k  \cup  M_{k/2})} \left( \prod_{i=1}^{k} b_{i\,\sigma(i)}\right) 
\end{eqnarray*}
where  $M_{k/2}$ is the {set of permutations corresponding to the maximum matchings of $G[\alpha]$} (i.e., a disjoint product of transpositions) and $D_k$ is the set of all derangements on $\alpha$.
Observe that for $\sigma \in  D_k \setminus M_{k/2}$, $\sigma$ must have a cycle of size 3 or more; however, $G$ is a tree, so no such $\sigma$ contributes to its sum.
Similarly, as in the proof of Lemma \ref{lem.skewnoodd}, any permutation $\sigma \not \in D_k$ also contributes 0. Therefore, we have
%Let $\mathcal{M}_{k/2}$ denote all matchings of size $k/2$. We have
\begin{eqnarray*}
\per(B) &=& \sum_{\sigma\in M_{k/2}} \left( \prod_{i=1}^{k} b_{i\,\sigma(i)}\right)\\ 
&=&(-1)^{k/2} \sum_{m \in M_{k/2}} \left( \prod_{\{i,j\} \in m} b_{ij}^2 \right).
\end{eqnarray*}
where the final line considers the matchings as a collection of edges.
Since for the last term, $b_{i,j}^2 > 0$, the final sum is nonzero so long as the sum is not empty. The sum is empty only when $k > 2\mu(G)$.

Now, we construct a skew-symmetric matrix $A$ whose underlying graph is a tree $T$ and $\ppr(A) = r_0 r_1 \ldots r_n$, where $r_j = 1$ if and only if $j$ is even and $1 \leq j \leq 2m$, for some $m\leq n/2$. Consider a path of length $2m$ on vertices $1,2,\ldots, 2m$. Add $n-2m$ vertices and connect all of them to vertex $2m-1$. Let $B$ be the adjacency matrix of this graph, and $A$ be the matrix obtained from $B$ by negating all the lower-diagonal entries. Since $T$ does not have any cycles, all the nonzero terms in the permanent of a principal submatrix of $A$ come from a matching of $T$. Hence $\ppr(A) = r_0 r_1 \ldots r_n$, with $r_j = 1$ if and only if $j$ is even and $1 \leq j \leq 2m$.
\end{proof}

%\section{Graph Products}\label{graphproducts}
%
%\begin{prop} \label{prop.graphsquare}
%
%Let $A,B \in M_n$ be nonnegative with $\ppr(A)=r_0,r_1,\dots,r_n$, $\ppr(B)=s_0,s_1,\dots,s_n$, $\ppr(A \otimes B)=q_0)q_1,\dots,q_{n^2}$.
%%, and $\ppr(A \otimes A)=t_0)t_1,\dots,t_{n^2}$. Then:
%
%\begin{enumerate}
%\item $q_0=1 \iff r_0s_0 = 1$.
%\item $q_1=1 \iff r_1s_1=1$.
%\item $r_k=s_l=1 \implies q_{k+l}=1$ for $k,l\geq 1$.
%%\item $r_{k}=1 \implies t_{k^2}=1$ for $k\geq 1$.
%%\item $t_{n^2}=1 \iff t_n=1$.
%\end{enumerate}
%\end{prop}
%
%\begin{proof}
%Another proof from the book goes here.
%\end{proof}

%The previous results focused on the permanent rank of nonnegative matrices. The non negativity a however skew-symmetric

%\section{Other Results}\label{sectionotherresults}
%\comment{- We have some partial results on Kronecker product of graphs.\\
%- ppr of a cone on G when ppr of G is known, is done.\\
%- ppr of skew symmetric matrices, what do we know??}
%\section{Future Work}\label{sectionfuture}
%\comment{What shall be done next?}

\section*{Acknowledgement}
An acknowledgment will be added after the work is accepted.

%The idea of the ppr-sequences was developed during a graduate research workshop hosted by the
%Denver University and the University of Colorado - Denver. We thank DU and UCD for their support and amazing hospitality. \comment{Can someone write a better thing here, please? I'm horrible at this things! like many other things.}

\end{document}